\RequirePackage{fix-cm}
\documentclass[twocolumn]{svjour3}          
\smartqed  
\usepackage{latexsym,amsfonts,amsmath,amssymb,mathrsfs,url}
\usepackage{tikz}

\newtheorem{algorithm}{Algorithm}
\usepackage{natbib}
 \bibpunct[, ]{(}{)}{,}{a}{}{,}%
 \def\bibhang{24pt}%
 \def\BIBand{and}%

\newcommand{\rd}{\, \mathrm{d}}
\newcommand{\MC}{\mathrm{MC}}
\newcommand{\TMC}{\mathrm{TMC}}
\newcommand{\Alg}{\mathrm{Alg}}
\newcommand{\bsx}{\boldsymbol{x}}
\newcommand{\tbsx}{\tilde{\boldsymbol{x}}}
\newcommand{\bsy}{\boldsymbol{y}}
\newcommand{\bsgamma}{\boldsymbol{\gamma}}
\newcommand{\bsmu}{\boldsymbol{\mu}}
\newcommand{\RR}{\mathbb{R}}
\newcommand{\NN}{\mathbb{N}}
\newcommand{\EE}{\mathbb{E}}
\newcommand{\VV}{\mathbb{V}}
\newcommand{\ZZ}{\mathbb{Z}}
\newcommand{\Fcal}{\mathcal{F}}
\newcommand{\Ncal}{\mathcal{N}}
\allowdisplaybreaks

\begin{document}

\title{Toeplitz Monte Carlo}

\author{Josef Dick \and Takashi Goda \and Hiroya Murata}

\authorrunning{Dick, Goda, Murata} 

\institute{J. Dick \at School of Mathematics and Statistics, University of New South Wales, Sydney NSW 2052, Australia \\ \email{josef.dick@unsw.edu.au}  \and
           T. Goda \and H. Murata \at School of Engineering, University of Tokyo, 7-3-1 Hongo, Bunkyo-ku, Tokyo 113-8656, Japan\\ \email{goda@frcer.t.u-tokyo.ac.jp} \\ \email{mh22151210@gmail.com}}

\date{Received: date / Accepted: date}

\maketitle

\begin{abstract}
Motivated mainly by applications to partial differential equations with random coefficients, 
we introduce a new class of Monte Carlo estimators, called Toeplitz Monte Carlo (TMC) estimator 
for approximating the integral of a multivariate function with respect to 
the direct product of an identical univariate probability measure. 
The TMC estimator generates a sequence $x_1,x_2,\ldots$ of i.i.d.\ samples 
for one random variable, and then uses $(x_{n+s-1},x_{n+s-2}\ldots,x_n)$ with $n=1,2,\ldots$
 as quadrature points, where $s$ denotes the dimension. 
Although consecutive points have some dependency, 
the concatenation of all quadrature nodes is represented by a Toeplitz matrix, 
which allows for a fast matrix-vector multiplication. 
In this paper we study the variance of the TMC estimator and its dependence on the dimension $s$. 
Numerical experiments confirm the considerable efficiency improvement 
over the standard Monte Carlo estimator for applications to partial differential equations 
with random coefficients, particularly when the dimension $s$ is large. 
\end{abstract}

\section{Introduction}
The motivation of this research mainly comes from applications to 
uncertainty quantification for ordinary or partial differential equations with random coefficients. 
The problem we are interested in is to estimate an expectation (integral)
\[ I_{\rho_s}(f):=\int_{\Omega^s}f(\bsx)\rho_s(\bsx)\rd \bsx\quad \text{with} \quad \rho_s(\bsx)=\prod_{j=1}^{s}\rho(x_j), \]
for large $s$ with $\rho$ being the univariate probability density function defined over $\Omega\subseteq \RR$. 
In some applications, the integrand is of the form
\[ f(\bsx) = g(\bsx A), \]
for a matrix $A\in \RR^{s\times t}$ and a function $g\colon \RR^t\to \RR$, see \cite{DKLS15}. 
Here we note that $\bsx$ is defined as a row vector. 
Typically, $\rho$ is given by the uniform distribution on the unit interval $\Omega=[0,1]$, 
or by the standard normal distribution on the real line $\Omega=\RR$.

The standard Monte Carlo method approximates $I_{\rho_s}(f)$ as follows: 
we first generate a sequence of i.i.d.\ samples of the random variables $\bsx\sim \rho_s$:
\[ \bsx_1=(x_{1,1},\ldots,x_{s,1}), \bsx_2=(x_{1,2},\ldots,x_{s,2}),\ldots, \]
and then approximate $I_{\rho_s}(f)$ by
\begin{equation}\label{eq:mc} I^{\MC}_{\rho_s}(f; N)=\frac{1}{N}\sum_{n=1}^{N}f(\bsx_n)=\frac{1}{N}\sum_{n=1}^{N}g(\bsx_nA).\end{equation}
It is well known that 
\[ \EE[I^{\MC}_{\rho_s}(f; N)] = I_{\rho_s}(f) \]
and
\begin{equation}\label{MC:variance}
\VV[I^{\MC}_{\rho_s}(f; N)] =\frac{I_{\rho_s}(f^2)-(I_{\rho_s}(f))^2}{N}, 
\end{equation}
which ensures the canonical ``one over square root of $N$'' convergence.

Now let us consider a situation where computing $\bsx_n A$ for $n=1,\ldots,N$ 
takes a significant amount of time in the computation of $I^{\MC}_{\rho_s}(f; N)$. 
In general, if the matrix $A$ does not have any special structure such as circulant, Hankel, Toeplitz, or Vandermonde, then fast matrix-vector multiplication is not available 
and the computation of $I^{\MC}_{\rho_s}(f; N)$ requires $O(Nst)$ arithmetic operations. Some examples where a fast matrix-vector multiplication has been established are the following: In \cite{FKS18} the authors use $H$-matrices to obtain an approximation of a covariance matrix which also permits a fast matrix vector multiplication; In \cite{GKSW08} the authors show how a (partially) fast matrix vector product can be implemented for multi-asset pricing in finance; Brownian bridge and principle component analysis factorizations of the covariance matrix in finance also permit a fast matrix vector multiplication \citep{GKSW08}. Here we consider the case where either a fast matrix-vector product is not available, or one wants to avoid $H$-matrices and particular covariance factorizations, since we do not impose any restrictions on $A$.

In order to reduce this computational cost, we propose an alternative, novel Monte Carlo estimator in this paper. 
Instead of generating a sequence of i.i.d.\ samples of the vector $\bsx$, 
we generate a sequence of i.i.d.\ samples of a single random variable, denoted by $x_1,x_2,\ldots$, 
and then approximates $I_{\rho_s}(f)$ by
\begin{equation}\label{TMC:est}
I^{\TMC}_{\rho_s}(f; N)=\frac{1}{N}\sum_{n=1}^{N}f(\tbsx_n)
\end{equation}
with
\[ \tbsx_n=(x_{n+s-1},\ldots,x_{n}).\]
The computation of $I^{\TMC}_{\rho_s}(f; N)$ can be done as follows:
\begin{algorithm}\label{alg:tmc} For $N\in \ZZ_{>0}$, let $x_1,x_2,\ldots,x_{N+s-1}$ be $N+s-1$ i.i.d.\ samples of a random variable following $\rho$.
\begin{enumerate}
\item Define $X\in \RR^{N\times s}$ by
\[ X=\begin{pmatrix} x_s & x_{s-1} & x_{s-2} & \cdots & x_1 \\ x_{s+1} & x_s & x_{s-1} & \cdots & x_2 \\ x_{s+2} & x_{s+1} & x_s & \cdots & x_3 \\ \vdots & \vdots & \vdots & \ddots & \vdots \\ x_{N+s-2} & x_{N+s-3} & x_{N+s-4} & \cdots & x_{N-1} \\ x_{N+s-1} & x_{N+s-2} & x_{N+s-3} & \cdots & x_{N} \\\end{pmatrix} .\]
Note that $X$ is a Toeplitz matrix.
\item Compute 
\[ XA=Y=\begin{pmatrix} \bsy_1 \\ \vdots \\ \bsy_N \end{pmatrix} \in \RR^{N\times t}.\] 
\item Then $I^{\TMC}_{\rho_s}(f; N)$ is given by
\[ I^{\TMC}_{\rho_s}(f; N) = \frac{1}{N}\sum_{n=1}^{N}g(\bsy_n). \]
\end{enumerate}
\end{algorithm}
\noindent The idea behind introducing this algorithm comes from a recent paper by \cite{DKLS15} 
who consider replacing the point set used in the standard Monte Carlo estimator \eqref{eq:mc} 
with a special class of quasi-Monte Carlo point sets which permit a fast matrix-vector multiplication $\bsx_n A$ for $n=1,\ldots,N$. This paper considers a sampling scheme different from \cite{DKLS15} while still allowing for a fast matrix-vector multiplication.

When $s$ is quite large, say thousands or million, $N$ has to be set significantly smaller than $2^s$. 
Throughout this paper we consider the case where $N\approx s^\kappa$ for some $\kappa>0$. 
Since the matrix-vector multiplication between a Toeplitz matrix $X$ and 
each column vector of $A$ can be done with $O(N\log s)$ arithmetic operations 
by using the fast Fourier transform \citep{FJ05}, the matrix-matrix multiplication $XA$ 
appearing in the second item of Algorithm~\ref{alg:tmc} can be done with $O(tN\log s)$ arithmetic operations. 
This way the necessary computational cost can be reduced from $O(Nst)$ to $O(tN\log s)$, 
which is the major advantage of using $I^{\TMC}_{\rho_s}(f; N)$.

\begin{remark}\label{parallel}
Here we give some comments about memory requirements and parallel implementation of the estimators. 
For the standard Monte Carlo estimator, since each sample $\bsx_n$ is generated independently, the corresponding function values $f(\bsx_n)=g(\bsx_n A)$ can be computed in parallel.
A required size to keep one vector $\bsx_nA$ in memory until evaluating $g(\bsx_n A)$ is only of order $t$. 
Regarding the TMC estimator, let us assume that $N$ is a multiple of $s$, that is, $N=Ls$ with some $L\in \ZZ_{>0}$. For each $\ell\in \{1,\ldots,L\}$, define a Toeplitz submatrix
\[ X_{\ell}=\begin{pmatrix} x_{\ell s} & x_{\ell s-1} & \cdots & x_{(\ell-1)s+1} \\ x_{\ell s+1} & x_{\ell s} & \cdots & x_{(\ell-1)s+2} \\ \vdots & \vdots & \ddots & \vdots \\ x_{\ell s+s-1} & x_{\ell s+s-2} & \cdots & x_{\ell s} \\\end{pmatrix}\in \RR^{s\times s}. \] 
In fact, it is easy to see that $X=\left(X_1^{\top},\ldots,X_L^{\top}\right)^{\top}$ and 
\[ Y=\left((X_1A)^{\top},\ldots,(X_LA)^{\top}\right)^{\top}\in \RR^{N\times t}. \]
This clearly shows that each $X_{\ell}A$ can be computed in parallel by using the fast Fourier transform. 
Given that the matrix $X_{\ell}A\in \RR^{s\times t}$ has to be kept in memory until evaluating $g(\bsy_n)$ for $n=(\ell-1)s+1,\ldots, \ell s$, the required memory size of the TMC estimator is of order $st$.
\end{remark}

In this paper we call $I^{\TMC}_{\rho_s}(f; N)$ a \emph{Toeplitz Monte Carlo} (TMC) estimator of $I_{\rho_s}(f)$ 
as we rely on the Toeplitz structure of $X$ to achieve a faster computation.\footnote{If the sample nodes are 
given by $\tbsx_n=(x_{n},\ldots,x_{n+s-1})$ instead of $\tbsx_n=(x_{n+s-1},\ldots,x_{n})$, 
the matrix $X$ becomes a Hankel matrix, which also allows for a fast matrix-vector multiplication. 
Therefore we can call our proposal a \emph{Hankel Monte Carlo} (HMC) estimator instead. 
However, in the context of Monte Carlo methods, HMC often refers to the Hamiltonian Monte Carlo algorithm, and we would like to avoid duplication of the abbreviations by coining the name Toeplitz Monte Carlo.} 
In the remainder of this paper, we study the variance of the TMC estimator and its dependence on the dimension $s$, 
and also see practical efficiency of the TMC estimator by carrying out numerical experiments 
for applications from ordinary/partial differential equations with random coefficients.

\section{Theoretical results}
\subsection{Variance analysis}
In order to study the variance of $I^{\TMC}_{\rho_s}(f; N)$, we introduce the concept of the analysis-of-variance (ANOVA) decomposition of multivariate functions \citep{H48,KSWW10,S93}. In what follows, for simplicity of notation, we write $[1:s]=\{1,\ldots,s\}$. For a subset $u\subseteq [1:s]$, we write $-u:=[1:s]\setminus u$ and denote the cardinality of $u$ by $|u|$. Let $f$ be a square-integrable function, i.e., $I_{\rho_s}(f^2)<\infty$. Then $f$ can be decomposed into
\[ f(\bsx) = \sum_{u\subseteq [1:s]}f_u(\bsx_u), \]
where we write $\bsx_u=(x_j)_{j\in u}$ and each summand is defined recursively by $f_{\emptyset} = I_{\rho_s}(f)$ and 
\[ f_u(\bsx_u) = \int_{\Omega^{s-|u|}}f(\bsx)\rho_{s-|u|}(\bsx_{-u})\rd \bsx_{-u} -\sum_{v\subset u}f_v(\bsx_v)\]
for $\emptyset \neq u\subseteq [1:s]$. Regarding this decomposition of multivariate functions, the following properties hold. We refer to Lemmas~A.1 \& A.3 of \cite{Owenxx} for the proof of the case where $\rho$ is the uniform distribution over the unit interval $\Omega=[0,1]$.
\begin{lemma}\label{lem:anova}
With the notation above, we have:
\begin{enumerate}
\item For any non-empty $u\subseteq [1:s]$ and $j\in u$, 
\[ \int_{\Omega}f_u(\bsx_u) \rho(x_j) \rd x_j = 0. \]
\item For any $u,v\subseteq [1:s]$, 
\begin{align*} I_{\rho_s}(f_uf_v) & = \int_{\Omega^s}f_u(\bsx_u)f_v(\bsx_v)\rho_s(\bsx)\rd \bsx\\
&  = \begin{cases} I_{\rho_s}(f_u^2) & \text{if $u=v$,} \\ 0 & \text{otherwise.} \end{cases}
\end{align*}
\end{enumerate}
\end{lemma}
It follows from the second assertion of Lemma~\ref{lem:anova} that
\begin{align*}
 I_{\rho_s}(f^2) & = I_{\rho_s}\left( \sum_{u,v\subseteq [1:s]}f_uf_v\right) \\
 & = \sum_{u,v\subseteq [1:s]}I_{\rho_s}(f_uf_v) = \sum_{u\subseteq [1:s]}I_{\rho_s}(f_u^2). 
\end{align*}
This equality means that the variance of $f$ can be expressed as a sum of the variances of the lower-dimensional functions:
\begin{equation}\label{eq:anova} I_{\rho_s}(f^2)-(I_{\rho_s}(f))^2 = \sum_{\emptyset \neq u\subseteq [1:s]}I_{\rho_s}(f_u^2). \end{equation}

Using these facts, the variance of the TMC estimator $I^{\TMC}_{\rho_s}(f; N)$ can be analyzed as follows:
\begin{theorem}\label{thm:variance}
Let $N,s \in \mathbb{Z}_{\ge 2}$. Then we have
\[ \EE[I^{\TMC}_{\rho_s}(f; N)] = I_{\rho_s}(f)\]
and
\begin{align*}
 & \VV[I^{\TMC}_{\rho_s}(f; N)] = \VV[I^{\MC}_{\rho_s}(f; N)] \\
 & \quad +\frac{2}{N^2}\sum_{\ell=1}^{\min(s,N)-1}(N-\ell)\sum_{\emptyset \neq u\subseteq [1:s-\ell]}I_{\rho_{|u|}}(f_uf_{u+\ell}), 
\end{align*}
where we write $u+\ell=\{j+\ell \colon j\in u\}$.
\end{theorem}
\noindent Note that, in the theorem, we write
\[ I_{\rho_{|u|}}(f_uf_{u+\ell}) = \int_{\Omega^{|u|}}f_u(\bsx_u)f_{u+\ell}(\bsx_u)\rho_{|u|}(\bsx_u)\rd \bsx_u. \]
The readers should not be confused with
\[ I_{\rho_s}(f_uf_{u+\ell}) = \int_{\Omega^s}f_u(\bsx_u)f_{u+\ell}(\bsx_{u+\ell})\rho_s(\bsx)\rd \bsx=0. \]

\begin{proof}
The first assertion follows immediately from the linearity of expectation and the trivial equality $\EE[f(\tbsx_n)]=I_{\rho_s}(f)$. For the second assertion, by using the ANOVA decomposition of $f$, we have
\begin{align*}
& \VV[I^{\TMC}_{\rho_s}(f; N)] \\
& = \EE\left[(I^{\TMC}_{\rho_s}(f; N))^2\right] - \left(\EE[I^{\TMC}_{\rho_s}(f; N)]\right)^2 \\
& = \EE\left[(I^{\TMC}_{\rho_s}(f; N))^2\right] - \left(I_{\rho_s}(f)\right)^2 \\
& = \frac{1}{N^2}\sum_{m,n=1}^{N}\EE[ f(\tbsx_m)f(\tbsx_n)]-\left(I_{\rho_s}(f)\right)^2 \\
& = \frac{1}{N^2}\sum_{n=1}^{N}\EE[(f(\tbsx_n))^2]+\frac{2}{N^2}\sum_{\substack{m,n=1\\ m>n}}^{N}\EE[ f(\tbsx_m)f(\tbsx_n)] \\
& \quad -\left(I_{\rho_s}(f)\right)^2 \\
& = \frac{I_{\rho_s}(f^2)}{N}+\frac{2}{N^2}\sum_{\substack{m,n=1\\ m>n}}^{N}\sum_{u,v\subseteq [1:s]}\EE[ f_u(\tbsx_{m,u})f_v(\tbsx_{n,v})]\\
& \quad -\left(I_{\rho_s}(f)\right)^2.
\end{align*}
It follows from the first assertion of Lemma~\ref{lem:anova} that the second term on the right-most side above becomes
\begin{align*}
& \frac{2}{N^2}\sum_{\substack{m,n=1\\ m>n}}^{N}\sum_{u,v\subseteq [1:s]}\EE[ f_u(\tbsx_{m,u})f_v(\tbsx_{n,v})] \\
& = \frac{2}{N^2}\sum_{\substack{m,n=1\\ m>n}}^{N}f_{\emptyset}^2 +\frac{2}{N^2}\sum_{\substack{m,n=1\\ m>n}}^{N}\sum_{\emptyset \neq u\subseteq [1:s]}f_\emptyset\, \EE[ f_u(\tbsx_{m,u})] \\
& \quad + \frac{2}{N^2}\sum_{\substack{m,n=1\\ m>n}}^{N}\sum_{\emptyset \neq v\subseteq [1:s]}f_\emptyset\, \EE[ f_v(\tbsx_{n,v})]\\
& \quad +\frac{2}{N^2}\sum_{\substack{m,n=1\\ m>n}}^{N}\sum_{\emptyset \neq u,v\subseteq [1:s]}\EE[ f_u(\tbsx_{m,u})f_v(\tbsx_{n,v})] \\
& = \frac{N-1}{N}\left(I_{\rho_s}(f)\right)^2\\
& \quad +\frac{2}{N^2}\sum_{\substack{m,n=1\\ m>n}}^{N}\sum_{\emptyset \neq u,v\subseteq [1:s]}\EE[ f_u(\tbsx_{m,u})f_v(\tbsx_{n,v})] \\
& = \frac{N-1}{N}\left(I_{\rho_s}(f)\right)^2\\ 
& \quad +\frac{2}{N^2}\sum_{\ell=1}^{N-1}\sum_{\substack{m,n=1\\ m-n=\ell}}^{N}\sum_{\emptyset \neq u,v\subseteq [1:s]}\EE[ f_u(\tbsx_{m,u})f_v(\tbsx_{n,v})] ,
\end{align*}
where we reordered the sum over $m$ and $n$ with respect to the difference $m-n$ in the last equality.

If $m-n\geq s$, there is no overlapping of the components between $\tbsx_m$ and $\tbsx_n$. Because of the independence of samples, it follows from the first assertion of Lemma~\ref{lem:anova} that the inner sum over $u$ and $v$ above is given by
\begin{align*}
& \sum_{\emptyset \neq u,v\subseteq [1:s]}\EE[ f_u(\tbsx_{m,u})f_v(\tbsx_{n,v})] \\
& = \sum_{\emptyset \neq u,v\subseteq [1:s]}\EE[ f_u(\tbsx_{m,u})]\EE[f_v(\tbsx_{n,v})] =0.
\end{align*}
If $\ell=m-n<s$, on the other hand, we have $\tilde{x}_{n,j}=\tilde{x}_{m,j+\ell}$ for any $j=1,\ldots,s-\ell$. With this equality and the first assertion of Lemma~\ref{lem:anova}, the inner sum over $u$ and $v$ becomes
\begin{align*}
& \sum_{\emptyset \neq u,v\subseteq [1:s]}\EE[ f_u(\tbsx_{m,u})f_v(\tbsx_{n,v})] \\
& = \sum_{\emptyset \neq u\subseteq [1:s-\ell]}\EE[f_{u+\ell}(\tbsx_{m,u+\ell})f_{u}(\tbsx_{n,u})] \\
& = \sum_{\emptyset \neq u\subseteq [1:s-\ell]}\EE[f_{u+\ell}(\tbsx_{n,u})f_{u}(\tbsx_{n,u})] \\
& = \sum_{\emptyset \neq u\subseteq [1:s-\ell]}I_{\rho_{|u|}}(f_{u+\ell}f_u).
\end{align*}

Altogether we obtain
\begin{align*}
& \VV[I^{\TMC}_{\rho_s}(f; N)] \\
& = \frac{I_{\rho_s}(f^2)}{N}+\frac{N-1}{N}\left(I_{\rho_s}(f)\right)^2-\left(I_{\rho_s}(f)\right)^2\\
& \quad +\frac{2}{N^2}\sum_{\ell=1}^{N-1}\sum_{\substack{m,n=1\\ m-n=\ell}}^{N}\sum_{\emptyset \neq u,v\subseteq [1:s]}\EE[ f_u(\tbsx_{m,u})f_v(\tbsx_{n,v})] \\
& = \frac{I_{\rho_s}(f^2)-\left(I_{\rho_s}(f)\right)^2}{N}\\
& \quad +\frac{2}{N^2}\sum_{\ell=1}^{\min(s,N)-1}\sum_{\emptyset \neq u\subseteq [1:s-\ell]}I_{\rho_{|u|}}(f_uf_{u+\ell})\sum_{\substack{m,n=1\\ m-n=\ell}}^{N}1\\
& = \VV[I^{\MC}_{\rho_s}(f; N)] \\
& \quad +\frac{2}{N^2}\sum_{\ell=1}^{\min(s,N)-1}(N-\ell)\sum_{\emptyset \neq u\subseteq [1:s-\ell]}I_{\rho_{|u|}}(f_uf_{u+\ell}).
\end{align*}
Thus we are done.\smartqed
\end{proof}

\begin{remark}
We now consider a parallel implementation of the TMC method with $L$ CPUs. Each CPU independently generates a Toeplitz matrix of $N$ random numbers and CPU $v$ computes an estimator $I^{\TMC}_{\rho_s}(f; N; v)$ of the form \eqref{TMC:est},  $v = 1, 2, \ldots, L$. In this case, Theorem~\ref{thm:variance} then applies to the output of each CPU in the parallel implementation. The average of these $L$ independent estimators is again unbiased and satisfies 
\begin{align*}
 & \VV\left[ \frac{1}{L} \sum_{v=1}^L I^{\TMC}_{\rho_s}(f; N; v)\right] = \frac{1}{L} \VV[I^{\MC}_{\rho_s}(f; N)] \\
 & \quad +\frac{2}{L N^2}\sum_{\ell=1}^{\min(s,N)-1}(N-\ell)\sum_{\emptyset \neq u\subseteq [1:s-\ell]}I_{\rho_{|u|}}(f_uf_{u+\ell}).
\end{align*}
Notice that we have $LN$ points altogether.

If $L=1$ then we get the result from Theorem~\ref{thm:variance} and if $N =1$ we get \eqref{MC:variance} (where, in this case, $L$ is the number of samples), as the sum $\sum_{\ell=1}^0 \ldots$ is equal to 0 in this instance.
\end{remark}

As is clear from Theorem~\ref{thm:variance}, the TMC estimator is unbiased and maintains the canonical ``one over square root of $N$'' convergence. Moreover, the TMC estimator can be regarded as a variance reduction technique since the second term on the variance $\VV[I^{\TMC}_{\rho_s}(f; N)]$ can be negative, depending on the function. 

\begin{example}
To illustrate the last comment, let us consider a simple test function $f\colon \RR^3\to \RR$ given by
\[ f(x,y,z) =x- y- z+xy-xz-yz, \]
and let $x,y,z$ be normally distributed independent random variables with mean 0 and variance 1. It is easy to see that
\begin{align*}
& f_{\{1\}}(x) = x, \quad f_{\{2\}}(y)=-y,\quad f_{\{3\}}(z)=-z, \\
& f_{\{1,2\}}(x,y) = xy, \quad f_{\{1,3\}}(x,z)=-xz, \\
& f_{\{2,3\}}(y,z)=-yz, \quad f_{\{1,2,3\}}(x,y,z) = 0.
\end{align*}
Then it follows that
\[ \VV[I^{\MC}_{\rho_s}(f; N)] =\frac{1}{N}\sum_{\emptyset \neq u\subseteq [1:3]}I_{\rho_s}(f_u^2) =\frac{6}{N}, \]
whereas, for $N\geq 3$, we have
\begin{align*}
 & \VV[I^{\TMC}_{\rho_s}(f; N)] \\
 & = \frac{6}{N} + \frac{2(N-2)}{N^2}I_{\rho_1}(f_{\{1\}}f_{\{3\}}) +\frac{2(N-1)}{N^2}\\
 & \quad \times \left[I_{\rho_1}(f_{\{1\}}f_{\{2\}})+I_{\rho_1}(f_{\{2\}}f_{\{3\}})+I_{\rho_2}(f_{\{1,2\}}f_{\{2,3\}})\right] \\
 & = \frac{6}{N}-\frac{2(N-2)}{N^2}-\frac{2(N-1)}{N^2} = \frac{2}{N}+\frac{6}{N^2}.
\end{align*}
Therefore the variance of the TMC estimator is almost one-third of the variance of the standard Monte Carlo estimator.

It is also possible that the variance of the TMC estimator increases compared to standard Monte Carlo, however, we show below that this increase is bounded.
\end{example}

\subsection{Weighted $L_2$ space and tractability}
Here we study the dependence of the variance $\VV[I^{\TMC}_{\rho_s}(f; N)]$ on the dimension $s$. 
For this purpose, we first give a bound on $\VV[I^{\TMC}_{\rho_s}(f; N)]$. For $1\leq \ell\leq s$, let 
\[ \alpha_\ell(f) := \left(\sum_{\substack{\emptyset \neq u\subseteq [1:s]\\ \min_{j\in u}j=\ell}}I_{\rho_{|u|}}(f_u^2) \right)^{1/2}.\]
Then it follows from the decomposition \eqref{eq:anova} that
\[ I_{\rho_s}(f^2)-(I_{\rho_s}(f))^2 = \sum_{\ell=1}^{s}\sum_{\substack{\emptyset \neq u\subseteq [1:s]\\ \min_{j\in u}j=\ell}}I_{\rho_s}(f_u^2) = \sum_{\ell=1}^{s}(\alpha_\ell(f))^2,\]
resulting in an equality
\[ \VV[I^{\MC}_{\rho_s}(f; N)]  = \frac{1}{N}\sum_{\ell=1}^{s}(\alpha_\ell(f))^2. \]
Using Theorem~\ref{thm:variance}, the variance $\VV[I^{\TMC}_{\rho_s}(f; N)]$ is bounded above as follows.
\begin{corollary}\label{cor:bound}
We have
\[ \VV[I^{\TMC}_{\rho_s}(f; N)] \leq \frac{1}{N}\left(\sum_{\ell=1}^{s}\alpha_\ell(f)\right)^2.\]
\end{corollary}
\begin{proof}
For any $\ell=1,\ldots,s-1$, the Cauchy-Schwarz inequality leads to
\begin{align*}
& \sum_{\emptyset \neq u\subseteq [1:s-\ell]}I_{\rho_{|u|}}(f_uf_{u+\ell}) \\
& \leq \sum_{\emptyset \neq u\subseteq [1:s-\ell]}\left(I_{\rho_{|u|}}(f_u^2)\right)^{1/2}\left(I_{\rho_{|u|}}(f_{u+\ell}^2)\right)^{1/2} \\
& = \sum_{v=1}^{s-\ell}\sum_{\substack{\emptyset \neq u\subseteq [1:s-\ell]\\ \min_{j\in u}j=v}}\left(I_{\rho_{|u|}}(f_u^2)\right)^{1/2}\left(I_{\rho_{|u|}}(f_{u+\ell}^2)\right)^{1/2}  \\
& \leq \sum_{v=1}^{s-\ell}\left( \sum_{\substack{\emptyset \neq u\subseteq [1:s-\ell]\\ \min_{j\in u}j=v}}I_{\rho_{|u|}}(f_u^2) \right)^{1/2} \\ 
& \qquad \times \left( \sum_{\substack{\emptyset \neq u\subseteq [1:s-\ell]\\ \min_{j\in u}j=v}}I_{\rho_{|u|}}(f_{u+\ell}^2)\right)^{1/2} \\
& = \sum_{v=1}^{s-\ell}\left( \sum_{\substack{\emptyset \neq u\subseteq [1:s-\ell]\\ \min_{j\in u}j=v}}I_{\rho_{|u|}}(f_u^2) \right)^{1/2} \\
& \qquad \times \left( \sum_{\substack{\emptyset \neq u\subseteq [\ell+1:s]\\ \min_{j\in u}j=\ell+v}}I_{\rho_{|u|}}(f_{u}^2)\right)^{1/2} \\
& \leq \sum_{v=1}^{s-\ell}\alpha_{v}(f)\alpha_{\ell+v}(f).
\end{align*}
Applying this bound to the second assertion of Theorem~\ref{thm:variance}, we obtain
\begin{align*}
& \VV[I^{\TMC}_{\rho_s}(f; N)] \\
& \leq \frac{1}{N}\sum_{\ell=1}^{s}(\alpha_\ell(f))^2 +\frac{2}{N^2}\sum_{\ell=1}^{s-1}(N-\ell)\sum_{v=1}^{s-\ell}\alpha_{v}(f)\alpha_{\ell+v}(f) \\
& \leq \frac{1}{N}\sum_{\ell=1}^{s}(\alpha_\ell(f))^2 +\frac{2}{N}\sum_{v=1}^{s-1}\alpha_{v}(f)\sum_{\ell=v+1}^{s}\alpha_{\ell}(f)\\
&  = \frac{1}{N}\left(\sum_{\ell=1}^{s}\alpha_\ell(f)\right)^2.
\end{align*}\smartqed
\end{proof}

Using this result, we have
\[ \frac{\VV[I^{\TMC}_{\rho_s}(f; N)]}{\VV[I^{\MC}_{\rho_s}(f; N)]}\leq \frac{(\alpha_1(f)+\cdots+\alpha_s(f))^2}{\alpha^2_1(f)+\cdots+\alpha^2_s(f)}\leq s, \]
wherein, for the second inequality, the equality is attained if and only if $\alpha_1(f)=\cdots=\alpha_s(f)$.
Therefore, when we fix the number of samples, the variance of the TMC estimator can at most be $s$ times larger than the variance of the standard Monte Carlo estimator.

Now let us consider the case $s=t$ and assume, as discussed in the first section, that the computational time for the standard Monte Carlo estimator is proportional to $Ns^2$, whereas the computational time for the TMC estimator is proportional to $Ns\log s$ (assuming that the main cost in evaluating $f(\bsx) = g(\bsx A )$ lies in the computation of $\bsx A$). When we fix the cost instead of the number of samples, we have
\[ N_{\MC} s^2 \asymp N_{\TMC}s\log s, \]
where $\asymp$ indicates that the terms should be of the same order, and so
\begin{align*} 
\frac{\VV[I^{\TMC}_{\rho_s}(f; N_{\TMC})]}{\VV[I^{\MC}_{\rho_s}(f; N_{\MC})]} & \leq \frac{N_{\MC}}{N_{\TMC}}\cdot \frac{(\alpha_1(f)+\cdots+\alpha_s(f))^2}{\alpha^2_1(f)+\cdots+\alpha^2_s(f)} \\
& \propto \log s. 
\end{align*} 
Thus, the variance of the TMC estimator for a given cost is at most $\log s$ times as large as the standard Monte Carlo estimator (up to some constant factor). On the other hand, if there is some decay of the importance of the ANOVA terms as the index of the variable increases, for instance, if the first few terms  in $\alpha_1(f) + \cdots + \alpha_s(f)$ dominate the sum, then the ratio $\frac{(\alpha_1(f)+\cdots+\alpha_s(f))^2}{\alpha^2_1(f)+\cdots+\alpha^2_s(f)}$ can be bounded independently of $s$, leading to a gain in the efficiency of the TMC estimator. We observe such a behaviour in our numerical experiments below.

Following the idea from \cite{SW98}, we now introduce the notion of a \emph{weighted} $L_2$ space. Let $(\gamma_u)_{u\subset \NN}$ be a sequence of the non-negative real numbers called weights. Then the weighted $L_2$ space is defined by
\[ \Fcal_{s,\bsgamma} = \left\{f\colon \Omega^s\to \RR \mid \|f\|_{s,\bsgamma}\leq \infty \right\}, \]
where
\[ \|f\|_{s,\bsgamma} := \left(\sum_{u\subseteq [1:s]}\gamma_u^{-1}I_{\rho_s}(f_u^2)\right)^{1/2}. \]
For any subset $u$ with $\gamma_u=0$, we assume that the corresponding ANOVA term $f_u$ is 0 and we formally set $0/0=0$. 

For a randomized algorithm using $N$ function evaluations of $f$ to estimate $I_{\rho_s}(f)$, which we denote by $\Alg(f;N)$, 
let us consider the minimal cost to estimate $I_{\rho_s}(f)$ with mean square error $\varepsilon^2$ for any $f$ in the unit ball of $\Fcal_{s,\bsgamma}$:
\[ N(\varepsilon, \Alg) := \min\left\{ N\in \ZZ_{>0} \mid e^2(\Alg, \Fcal_{s,\bsgamma} )\leq \varepsilon^2 \right\}, \]
where
\[ e^2(\Alg, \Fcal_{s,\bsgamma} ):=\sup_{\substack{f\in \Fcal_{s,\bsgamma} \\ \|f\|_{s,\bsgamma}\leq 1}}\EE\left[ \left( \Alg(f;N) -I_{\rho_s}(f)\right)^2 \right]. \]
We say that the algorithm $\Alg$ is
\begin{itemize}
\item a \emph{weakly tractable} algorithm if
\[ \lim_{\varepsilon^{-1}+s\to \infty}\frac{\ln N(\varepsilon, \Alg)}{\varepsilon^{-1}+s}=0, \]
\item a \emph{polynomially tractable} algorithm if there exist non-negative constants $C,p,q$, such that
\[ N(\varepsilon, \Alg)\leq C \varepsilon^{-p}s^q \]
holds for all $s=1,2,\ldots$, where $p$ and $q$ are called the $\varepsilon^{-1}$-exponent and the $s$-exponent, respectively,
\item a \emph{strongly polynomially tractable} algorithm if $\Alg$ is a polynomially tractable algorithm with the $s$-exponent 0.
\end{itemize}
We refer to \cite{NWbook} for more information on the notion of tractability.

For instance, the standard Monte Carlo estimator $I^{\MC}_{\rho_s}(f; N)$ is a strongly polynomially tractable algorithm with $\varepsilon^{-1}$-exponent 2 if
\begin{equation}\label{eq:mc_cond} \sup_{u\subset \NN}\gamma_u<\infty \end{equation}
holds. This claim can be proven as follows:
\begin{align*}
& \EE\left[ \left( I^{\MC}_{\rho_s}(f; N) -I_{\rho_s}(f)\right)^2 \right] \\
& = \VV[I^{\MC}_{\rho_s}(f; N)] = \frac{1}{N}\sum_{\emptyset \neq u\subseteq [1:s]}I_{\rho_s}(f_u^2) \\
& \leq \frac{1}{N}\left( \max_{u\subseteq [1:s]}\gamma_u \right)\left(\sum_{u\subseteq [1:s]}\frac{I_{\rho_s}(f_u^2)}{\gamma_u}\right) \\
& = \frac{\|f\|_{s,\bsgamma}^2}{N}\max_{u\subseteq [1:s]}\gamma_u .
\end{align*}
It follows that, in order to have 
\[ \EE\left[ \left( I^{\MC}_{\rho_s}(f; N) -I_{\rho_s}(f)\right)^2 \right]\leq \varepsilon^2 \]
for any $f\in \Fcal_{s,\bsgamma}$ with $\|f\|_{s,\bsgamma}\leq 1$, we need $N\geq  \varepsilon^{-2} \max_{u\subseteq [1:s]}\gamma_u$. Thus the minimal cost is bounded above by
\[ N(\varepsilon, I^{\MC}_{\rho_s}) \leq \varepsilon^{-2} \max_{u\subseteq [1:s]}\gamma_u . \]
Given the condition \eqref{eq:mc_cond}, we see that $N(\varepsilon, I^{\MC}_{\rho_s})$ is bounded independently of the dimension $s$ and the algorithm $I^{\MC}_{\rho_s}$ is strongly polynomially tractable with the $\varepsilon^{-1}$-exponent 2.

The following theorem gives the necessary conditions on the weights $(\gamma_u)_{u\subset \NN}$ for the TMC estimator to be a weakly tractable algorithm, a polynomially tractable algorithm, or a strongly polynomially tractable algorithm.
\begin{theorem}\label{thm:tractability}The TMC estimator is 
\begin{itemize}
\item a weakly tractable algorithm if
\[ \lim_{s\to \infty}\frac{1}{s}\ln \left(\sum_{\ell=1}^{s}\max_{\substack{u\subset [1:s]\\ \min_{j\in u}j=\ell}}\gamma_u\right) =0, \]
\item a polynomially tractable algorithm with the $\varepsilon^{-1}$-exponent 2 if there exists $q>0$ such that
\[ \sup_{s=1,2,\ldots}\frac{1}{s^q}\sum_{\ell=1}^{s}\max_{\substack{u\subset [1:s]\\ \min_{j\in u}j=\ell}}\gamma_u < \infty, \]
\item a strongly polynomially tractable algorithm with the $\varepsilon^{-1}$-exponent 2 if
\[ \sum_{\ell=1}^{\infty}\sup_{\substack{u\subset \NN \\ \min_{j\in u}j=\ell}}\gamma_u < \infty. \]
\end{itemize}
\end{theorem}

\begin{proof} It follows from Corollary~\ref{cor:bound} and H\"{o}lder's inequality for sums that
\begin{align*}
& \EE\left[ \left( I^{\TMC}_{\rho_s}(f; N) -I_{\rho_s}(f)\right)^2 \right] \\
& = \VV[I^{\TMC}_{\rho_s}(f; N)] \leq \frac{1}{N}\left(\sum_{\ell=1}^{s}\alpha_\ell(f)\right)^2 \\
& = \frac{1}{N}\left(\sum_{\ell=1}^{s}\left(\sum_{\substack{\emptyset \neq u\subseteq [1:s]\\ \min_{j\in u}j=\ell}}I_{\rho_{|u|}}(f_u^2) \right)^{1/2}\right)^2 \\
& \leq \frac{1}{N}\left(\sum_{\ell=1}^{s}\left(\max_{\substack{\emptyset \neq u\subseteq [1:s]\\ \min_{j\in u}j=\ell}}\gamma_u\right)^{1/2}\right.\\
& \quad \times \left. \left(\sum_{\substack{\emptyset \neq u\subseteq [1:s]\\ \min_{j\in u}j=\ell}}\gamma_u^{-1}I_{\rho_{|u|}}(f_u^2) \right)^{1/2}\right)^2\\
& \leq \frac{1}{N}\left( \sum_{\ell=1}^{s}\max_{\substack{\emptyset \neq u\subseteq [1:s]\\ \min_{j\in u}j=\ell}}\gamma_u\right)\left( \sum_{\ell=1}^{s}\sum_{\substack{\emptyset \neq u\subseteq [1:s]\\ \min_{j\in u}j=\ell}}\gamma_u^{-1}I_{\rho_{|u|}}(f_u^2) \right) \\
& \leq \frac{\|f\|_{s,\bsgamma}^2}{N}\sum_{\ell=1}^{s}\max_{\substack{\emptyset \neq u\subseteq [1:s]\\ \min_{j\in u}j=\ell}}\gamma_u .
\end{align*}
Thus, the minimal cost to have 
\[ \EE\left[ \left( I^{\TMC}_{\rho_s}(f; N) -I_{\rho_s}(f)\right)^2 \right]\leq \varepsilon^2\]
for any $f\in \Fcal_{s,\bsgamma}$ with $\|f\|_{s,\bsgamma}\leq 1$ is bounded above by
\[ N(\varepsilon, I^{\TMC}_{\rho_s}) \leq \varepsilon^{-2} \sum_{\ell=1}^{s}\max_{\substack{\emptyset \neq u\subseteq [1:s]\\ \min_{j\in u}j=\ell}}\gamma_u . \]

Let us consider the first assertion of the theorem. If the weights satisfy
\[ \lim_{s\to \infty}\frac{1}{s}\ln \left(\sum_{\ell=1}^{s}\max_{\substack{u\subset [1:s]\\ \min_{j\in u}j=\ell}}\gamma_u\right) =0, \]
then we have
\begin{align*}
& \lim_{\varepsilon^{-1}+s\to \infty}\frac{\ln N(\varepsilon, I^{\TMC}_{\rho_s})}{\varepsilon^{-1}+s} \\
& \leq \lim_{\varepsilon^{-1}+s\to \infty}\left(\frac{\ln \varepsilon^{-2}}{\varepsilon^{-1}} +\frac{1}{s}\ln \left(\sum_{\ell=1}^{s}\max_{\substack{\emptyset \neq u\subseteq [1:s]\\ \min_{j\in u}j=\ell}}\gamma_u\right) \right) =0,
\end{align*}
meaning that $I^{\TMC}_{\rho_s}$ is a weakly tractable algorithm. Since the second and third assertions can be shown similarly, we omit the proof.\smartqed
\end{proof}

For instance, if the weights satisfy $\gamma_u\geq \gamma_v$ whenever $u\subset v$, we always have
\[ \sup_{\substack{u\subset \NN \\ \min_{j\in u}j=\ell}}\gamma_u = \gamma_{\{\ell\}}\]
for any $\ell \in \ZZ_{>0}$. Therefore, the necessary condition for the TMC estimator to be strongly polynomially tractable reduces to a simple summability:
\[ \sum_{\ell=1}^{\infty}\gamma_{\{\ell\}} < \infty. \]
It is obvious to see that the necessary condition for the TMC estimator to be weakly tractable is stronger than that for the standard Monte Carlo estimator to be strongly tractable.
Whether we can weaken the necessary conditions for the TMC estimator given in Theorem~\ref{thm:tractability} or not is an open question.

\section{Numerical experiments}
In order to see the practical performance of the TMC estimator, we conduct four kinds of numerical experiments.\footnote{The C codes used in our experiments are available from \url{https://github.com/takashigoda/Toeplitz-Monte-Carlo}.} 
The first test case follows Section~4.1 of \cite{DKLS15}, which considers generating quadrature points 
from a multivariate normal distribution with a general covariance matrix.
The second test case, taken from Section~4.2 of \cite{DKLS15}, 
deals with approximating linear functionals of solutions of one-dimensional PDE with ``uniform'' random coefficients. 
The third test case is an extension of the second test case to a one-dimensional PDE with ``log-normal'' random coefficients.
Finally, in the fourth test case we consider another possible extension of the second test case, namely a two-dimensional PDE with uniform random coefficients.
All computations are performed on a laptop with 1.6 GHz Intel Core i5 CPU and 8 GB memory.

For every test case, we carry out numerical experiments with various values of $N$ and $s$ 
using both the standard Monte Carlo estimator and the TMC estimator.
For each pair of $N$ and $s$, we repeat computations $R=25$ times independently  
and calculate the average computational time. For the latter three test cases, the variances of these estimators are measured by
\[ \frac{1}{R(R-1)}\sum_{r=1}^{R}\left( I^{\bullet, (r)}_{\rho_s}(f; N)-\overline{I^{\bullet}_{\rho_s}(f; N)}\right)^2 \]
with
\[ \overline{I^{\bullet}_{\rho_s}(f; N)}=\frac{1}{R}\sum_{r=1}^{R}I^{\bullet, (r)}_{\rho_s}(f; N), \]
for $\bullet\in \{\MC,\TMC\}$, where $I^{\bullet, (r)}_{\rho_s}(f; N)$ denotes the $r$-th realization of the estimator $I^{\bullet}_{\rho_s}(f; N)$.

\subsection{Generating points from multivariate Gaussian}
\begin{table*}
\caption{Average times (in seconds) to generate normally distributed points with a random covariance matrix 
for various values of $N$ and $s$ using the standard Monte Carlo method and the TMC method.}\label{tbl:mvn}
\begin{center}
\begin{tabular}{|c|c|c|c|c|c|c|}
 \hline
 & $N$ & $s=128$ & $s=256$ & $s=512$ & $s=1024$ & $s=2048$ \\ \hline \hline
 stdMC & 1024  & 0.041 & 0.192 & 1.369 & 9.613 & -- \\ 
 TMC &         & 0.165 & 0.239 & 0.403 & 0.872 & -- \\ \hline
 saving &       & 0.250 & 0.800 & 3.399 & 11.021 & -- \\ \hline \hline
 stdMC & 2048  & 0.083 & 0.367 & 2.077 & 20.154 & 217.162 \\ 
 TMC &         & 0.322 & 0.465 & 0.923 & 1.794 & 3.599 \\ \hline
 saving &       & 0.259 & 0.790 & 2.251 & 11.234 & 60.342 \\ \hline \hline
 stdMC & 4096  & 0.154 & 0.740 & 5.551 & 45.287 & 446.974 \\ 
 TMC &         & 0.586 & 0.981 & 1.909 & 3.407 & 7.112 \\ \hline
 saving &       & 0.263 & 0.754 & 2.908 & 13.291 & 62.851 \\ \hline \hline
 stdMC & 8192  & 0.288 & 1.669 & 10.332 & 87.988 & 830.648 \\ 
 TMC &         & 1.169 & 1.868 & 3.110 & 6.204 & 14.561 \\ \hline
 saving &       & 0.247 & 0.893 & 3.322 & 14.183 & 57.047 \\ \hline \hline
 stdMC & 16384 & 0.660 & 3.255 & 19.764 & 159.792 & 1820.426 \\ 
 TMC &         & 2.434 & 3.750 & 7.088 & 14.073 & 28.082 \\ \hline
 saving &       & 0.271 & 0.868 & 2.788 & 11.355 & 64.826 \\ \hline \hline
 stdMC & 32768 & 1.078 & 5.627 & 33.011 & 331.747 & 3648.755 \\ 
 TMC &         & 5.486 & 9.098 & 15.164 & 25.950 & 57.171 \\ \hline
 saving &       & 0.197 & 0.618 & 2.177 & 12.784 & 63.822 \\ \hline 
\end{tabular}
\end{center}
\end{table*}

Generating quadrature points from the multivariate normal distribution $\Ncal(\bsmu,\Sigma)$ 
with mean vector $\bsmu\in \RR^s$ and covariance matrix $\Sigma\in \RR^{s\times s}$ is ubiquitous in scientific computation.
The standard procedure is as follows \citep[Chapter~XI.2]{Dev86}: Let $A\in \RR^{s\times s}$ be a matrix which satisfies $A^{\top} A=\Sigma$.
For instance, the Cholesky decomposition gives such $A$ in an upper triangular form for any symmetric positive-definite matrix $\Sigma$.
Using this decomposition, we can generate a point $\bsy\sim \Ncal(\bsmu,\Sigma)$ by first generating $\bsx=(x_1,\ldots,x_s)$ with $x_j\sim \Ncal(0,1)$ and then transforming $\bsx$ by
\[ \bsy = \bsmu + \bsx A. \]
Even if the matrix $A$ does not have any further structure, a set of quadrature points can be generated in a fast way by following Algorithm~\ref{alg:tmc}.

For our experiments, we fix $\bsmu=(0,\ldots,0)$ and choose $A$ randomly such that $A$ is a random upper triangular matrix with positive diagonal entries.
Table~\ref{tbl:mvn} shows the average computational times for various values of $N$ and $s$.
As the theory predicts, the computational time for the standard Monte Carlo (stdMC) estimator scales as $Ns^2$, whereas that for the TMC estimator it approximately scales as $Ns\log s$. 
For low-dimensional cases up to $s=256$, the stdMC estimator is faster to compute than the TMC estimator. However, as the dimension $s$ increases, the relative speed of the TMC estimator also increases. For the case $s=2048$, the TMC is approximately 60 times faster than the stdMC.
This result indicates that the TMC estimator is useful in high-dimensional settings for generating normally distributed quadrature points for a general covariance matrix.

\subsection{1D differential equation with uniform random coefficients}
Let us consider the ODE
\begin{align*}
& -\frac{\mathrm{d}}{\mathrm{d}x}\left( a(x,\bsy)\frac{\mathrm{d}}{\mathrm{d}x}u(x,\bsy)\right) = g(x)\equiv 1 \\
& \qquad \qquad \qquad \text{for $x\in (0,1)$ and $\bsy\in \left[ -\frac{1}{2}, \frac{1}{2}\right]^{\NN}$} \\
& u(x,\bsy) = 0\quad \text{for $x=0,1$} \\
& a(x,\bsy) = 2+\sum_{j=1}^{\infty}y_j\frac{\sin(2\pi jx)}{j^{3/2}}.
\end{align*}
In order to solve this ODE approximately with the finite element method, 
we consider a system of hat functions
\[ \phi_m(x) = \begin{cases} (x-x_{m-1})M & \text{if $x_{m-1}\leq x\leq x_m$,} \\ (x_{m+1}-x)M & \text{if $x_m\leq x\leq x_{m+1}$,} \\ 0 & \text{otherwise,}\end{cases} \]
for $m=1,\ldots,M-1$ over $M+1$ equi-distributed nodes $x_m=m/M$ for $m=0,1,\ldots,M$, 
and truncate the infinite sum appearing in the random field $a$ by the first $s$ terms.
Therefore, we have three different parameters $N, M$ and $s$.

Given the boundary condition on $u$, the approximate solution $u_M$ of the ODE for $y_1,\ldots,y_s\sim U[-1/2, 1/2]$ is given by
\[ u_M = \sum_{m=1}^{M-1}\hat{u}_m \phi_m(x), \]
with
\begin{equation}\label{eq:linear_system1} \left( \hat{u}_1,\ldots,\hat{u}_{M-1}\right)\cdot B(y_1,\ldots,y_s) = (\hat{g}_1,\ldots,\hat{g}_{M-1}), \end{equation}
for the symmetric stiffness matrix $B$ depending on $y_1,\ldots,y_s$ and the forcing vector with entries $\hat{g}_m=\int_0^1 g(x)\phi_m(x)\rd x=1/M$.
The entries of the matrix $B=(b_{k,\ell})_{k,\ell}\in \RR^{(M-1)\times (M-1)}$ are given by
\begin{align*}
 b_{k,\ell} & = \int_0^1 a(x,(y_1,\ldots,y_s,0,0,\ldots)){\phi}'_k(x){\phi}'_{\ell}(x)\rd x \\
 & = 2\int_0^1 {\phi}'_k(x){\phi}'_{\ell}(x)\rd x \\
 & \quad + \sum_{j=1}^{s}y_j\int_0^1 \frac{\sin(2\pi jx)}{j^{3/2}}{\phi}'_k(x){\phi}'_{\ell}(x)\rd x \\
 & =: a_{k,\ell}^{(0)}+\sum_{j=1}^{s}y_j a_{k,\ell}^{(j)}.
\end{align*}
Hence, by letting $A^{(j)} = (a_{k,\ell}^{(j)})_{k,\ell}\in \RR^{(M-1)\times (M-1)}$ for $j=0,1,\ldots,s$, we have
\[ B=A^{(0)}+\sum_{j=1}^{s}y_jA^{(j)}. \]

Here we note that every entry of $A^{(j)}$ can be explicitly calculated as
\[ a_{k,\ell}^{(0)}=\begin{cases} 4M & \text{if $k=\ell$,}\\ -2M & \text{if $|k-\ell|=1$,} \\ 0 & \text{otherwise,} \end{cases} \] 
and 
\[ a_{k,\ell}^{(j)}=\begin{cases} \frac{M^2}{\pi j^{5/2}}\sin \left( \frac{2\pi j}{M}\right)\sin\left( \frac{2\pi jk}{M}\right) & \text{if $k=\ell$,}\\- \frac{M^2}{\pi j^{5/2}}\sin \left( \frac{\pi j}{M}\right)\sin\left( \frac{\pi j(k+\ell)}{M}\right) & \text{if $|k-\ell|=1$,} \\ 0 & \text{otherwise,} \end{cases} \] 
for $1\leq j\leq s$.
Since each matrix $A^{(j)}$ is tridiagonal, the LU decomposition requires only $O(M)$ computational time to solve the system of linear equations \eqref{eq:linear_system1}. 
This way, it is clear that computing the matrix $B$ for $N$ Monte Carlo samples on $(y_1,\ldots,y_s)$ is computationally dominant for the whole process, 
and as shown in Section~3.2 of \cite{DKLS15}, the standard Monte Carlo method requires $O(NMs)$ arithmetic operations, 
whereas Algorithm~\ref{alg:tmc} can reduce them to $O(NM\log s)$.
It is expected that the TMC estimator brings substantial computational cost savings, particularly for large $s$.
For our experiments, we estimate the expectation of $u(1/2,\cdot)$.

Table~\ref{tbl:pde1} shows the results for various values of $N,M$ and $s$. In general, in the area of PDEs with random coefficients,
both $M$ and $s$ grow with $N$ so that the following three errors are balanced \citep{DKLS15}: the finite element discretization error, the truncation error on the random field $a$, the Monte Carlo error.
The optimal balancing depends on the decay rates of these errors. However, in our numerical experiments, we are only interested in how the computation times change for different relations between $M$, $s$, and $N$, and so we test different relations between those parameters (irrespective of what the optimal choice actually is).

Since computations are repeated 25 times independently for each pair, here we report the average of the estimation values and the variance of the estimator for two methods.
By comparing the mean values computed by two methods, we can confirm that the TMC estimator is also unbiased (just as the stdMC estimator is). The variances for both of the estimators decay with the rate $O(1/N)$, whereas the magnitude for the TMC estimator is approximately 2--5 larger than the stdMC estimator.
On the other hand, the computational time for the stdMC estimator increases with $N$ (equivalently, with $s$) significantly faster than the TMC estimator.
This increment behavior of the computation times indicates that computation of the stiffness matrix is the most computationally dominant part in this computation, 
and so the TMC estimator is quite effective in reducing the computation time.

\begin{table*}
\caption{Estimating the expectation of $u(1/2,\cdot)$ with various values of $N, M$ and $s$ using the standard Monte Carlo method and the TMC method for the uniform case. 
The average estimate, the variance of the estimator and the average computational time (in seconds) are shown for each method. 
The efficiency is defined by the ratio of the product of the variance and the computational time between two methods.}\label{tbl:pde1}
\begin{center}
\begin{tabular}{ |c|c|c|c|c|c|c|c| } 
 \hline
  & \multicolumn{3}{|c|}{stdMC} & \multicolumn{3}{|c|}{TMC} & \\ \hline \hline
  & mean & variance & time &  mean & variance & time & efficiency \\ \hline 
 $N$ & \multicolumn{7}{|c|}{$N=M=s$} \\ \hline 
 64 & 0.066 & 2.00$\cdot 10^{-8}$ & 0.002 & 0.066 & 4.39$\cdot 10^{-8}$ & 0.023 & 0.046 \\ 
 128 & 0.065 & 4.20$\cdot 10^{-9}$ & 0.020 & 0.065 & 1.47$\cdot 10^{-8}$ & 0.071 & 0.081 \\
 256 & 0.064 & 2.00$\cdot 10^{-9}$ & 0.189 & 0.064 & 7.11$\cdot 10^{-9}$ & 0.221 & 0.240 \\
 512 & 0.064 & 1.01$\cdot 10^{-9}$ & 1.455 & 0.064 & 2.27$\cdot 10^{-9}$ & 0.823 & 0.781 \\
1024 & 0.064 & 5.75$\cdot 10^{-10}$ & 19.720 & 0.064 & 8.45$\cdot 10^{-10}$ & 3.202 & 4.190 \\
2048 & 0.064 & 1.96$\cdot 10^{-10}$ & 217.303 & 0.064 & 5.51$\cdot 10^{-10}$ & 12.069 & 6.405 \\
4096 & 0.064 & 9.50$\cdot 10^{-11}$ & 2367.836 & 0.064 & 2.36$\cdot 10^{-10}$ & 57.402 & 16.612 \\ \hline
 $N$ & \multicolumn{7}{|c|}{$N=M^2=s$} \\ \hline 
256 & 0.076 & 3.34$\cdot 10^{-8}$ & 0.005 & 0.076 & 1.17$\cdot 10^{-7}$ & 0.013 & 0.125 \\
512 & 0.072 & 7.88$\cdot 10^{-9}$ & 0.042 & 0.072 & 2.44$\cdot 10^{-8}$ & 0.031 & 0.441 \\
1024 & 0.070 & 2.77$\cdot 10^{-9}$ & 0.298 & 0.070 & 4.98$\cdot 10^{-9}$ & 0.087 & 1.892 \\
2048 & 0.068 & 4.78$\cdot 10^{-10}$ & 1.871 & 0.068 & 2.07$\cdot 10^{-9}$ & 0.235 & 1.838 \\
4096 & 0.066 & 3.66$\cdot 10^{-10}$ & 10.965 & 0.066 & 6.48$\cdot 10^{-10}$ & 0.766 & 8.088 \\
8192 & 0.066 & 6.70$\cdot 10^{-11}$ & 75.863 & 0.066 & 2.28$\cdot 10^{-10}$ & 2.390 & 9.329 \\
16384 & 0.065 & 2.20$\cdot 10^{-11}$ & 963.569 & 0.065 & 9.70$\cdot 10^{-11}$ & 7.033 & 31.073 \\ \hline
 $N$ & \multicolumn{7}{|c|}{$N=2M=2s$} \\ \hline 
64 & 0.069 & 3.06$\cdot 10^{-8}$ & 0.001 & 0.070 & 9.98$\cdot 10^{-8}$ & 0.017 & 0.013 \\
128 & 0.066 & 6.16$\cdot 10^{-9}$ & 0.004 & 0.066 & 2.36$\cdot 10^{-8}$ & 0.044 & 0.026 \\
256 & 0.065 & 2.15$\cdot 10^{-9}$ & 0.041 & 0.065 & 9.24$\cdot 10^{-9}$ & 0.135 & 0.070 \\
512 & 0.064 & 7.61$\cdot 10^{-10}$ & 0.326 & 0.064 & 2.58$\cdot 10^{-9}$ & 0.449 & 0.214 \\
1024 & 0.064 & 3.86$\cdot 10^{-10}$ & 2.681 & 0.064 & 8.95$\cdot 10^{-10}$ & 1.413 & 0.818 \\
2048 & 0.064 & 1.72$\cdot 10^{-10}$ & 35.731 & 0.064 & 5.67$\cdot 10^{-10}$ & 5.648 & 1.919 \\
4096 & 0.064 & 7.84$\cdot 10^{-11}$ & 444.156 & 0.064 & 2.41$\cdot 10^{-10}$ & 23.480 & 6.144 \\ \hline
 $N$ & \multicolumn{7}{|c|}{$N=2M^2=2s$} \\ \hline 
512 & 0.076 & 1.39$\cdot 10^{-8}$ & 0.011 & 0.076 & 4.37$\cdot 10^{-8}$ & 0.026 & 0.137 \\
1024 & 0.072 & 3.25$\cdot 10^{-9}$ & 0.064 & 0.072 & 9.13$\cdot 10^{-9}$ & 0.062 & 0.364 \\
2048 & 0.070 & 7.20$\cdot 10^{-10}$ & 0.534 & 0.070 & 3.26$\cdot 10^{-9}$ & 0.183 & 0.645 \\
4096 & 0.068 & 2.23$\cdot 10^{-10}$ & 3.204 & 0.068 & 9.55$\cdot 10^{-10}$ & 0.524 & 1.427 \\
8192 & 0.066 & 1.41$\cdot 10^{-10}$ & 20.888 & 0.066 & 3.01$\cdot 10^{-10}$ & 1.596 & 6.133 \\
16384 & 0.066 & 3.60$\cdot 10^{-11}$ & 159.880 & 0.066 & 1.22$\cdot 10^{-10}$ & 4.618 & 10.216 \\
32768 & 0.065 & 1.02$\cdot 10^{-11}$ & 2023.179 & 0.065 & 5.80$\cdot 10^{-11}$ & 13.586 & 26.144 \\ \hline
\end{tabular}
\end{center}
\end{table*}

As is standard (see for instance Chapter~8 of \cite{Owenxx}), we measure the relative efficiency of the TMC estimator compared to the stdMC estimator by the ratio
\[ \frac{T_{\MC}\sigma^2_{\MC}}{T_{\TMC}\sigma^2_{\TMC}},\]
where $T_{\bullet}$ and $\sigma^2_{\bullet}$ denote the computational time spent and the variance, respectively, for the estimators $\bullet\in \{\MC,\TMC\}$.
As shown in the rightmost column of Table~\ref{tbl:pde1}, 
the relative efficiency is smaller than 1 for low-dimensional cases, 
which means that we do not gain any benefit from using the TMC estimator.
However, because of the substantial computational time savings, 
the efficiency increases significantly for large $s$ where it goes well beyond 1.

\subsection{1D differential equation in the log-normal case}
Let us move on to an ODE with the log-normal random coefficients:
\begin{align*}
& -\frac{\mathrm{d}}{\mathrm{d}x}\left( a(x,\bsy)\frac{\mathrm{d}}{\mathrm{d}x}u(x,\bsy)\right) = g(x)\equiv 1\\
& \qquad \qquad \qquad \text{for $x\in (0,1)$ and $y_j\sim \text{i.i.d.\ } N(0,1)$} \\
& u(x,\bsy) = 0\quad \text{for $x=0,1$} \\
& a(x,\bsy) = \exp\left(\sum_{j=1}^{\infty}y_j\frac{\sin(2\pi jx)}{j^{2}}\right) .
\end{align*}
Similarly to the uniform case, we truncate the infinite sum appearing in the random field $a$ by the first $s$ terms. A similar test case was also used in Section~4.3 of \cite{DKLS15}.\footnote{We point out that in Section~4.3 of \cite{DKLS15} the authors replaced the normal distribution by a uniform distribution. However, a normal distribution could have been used in \cite[Section~4.3]{DKLS15} as well, for instance by shifting the QMC points in $[0,1]^s$ in the quadrature rule by $(1/(2N), \ldots, 1/(2N))$ and applying the inverse CDF to the shifted points. Doing so avoids the point $(0,\ldots, 0) \in [0,1]^s$ which would get transformed to $(-\infty, \ldots, -\infty)$. Note that this shift does not effect the fast QMC matrix vector product, which can still be applied in the usual way.}

Now, as the entries of the stiffness matrix $B=(b_{k,\ell})_{k,\ell}\in \RR^{(M-1)\times (M-1)}$ cannot be expressed simply as a linear sum of $y_1,y_2,\ldots$,
we need to approximate the integral by using some quadrature formulas, except for the case $|k-\ell|\geq 2$ where we just have $b_{k,\ell}=0$.
Denoting the quadrature nodes and the corresponding weights by $x_{1,k,\ell},\ldots,x_{I,k,\ell}$ and $\omega_{1,k,\ell},\ldots,\omega_{I,k,\ell}$, the entry $b_{k,\ell}$ is approximated by
\[ b_{k,\ell} \approx \hat{b}_{k,\ell} = \sum_{i=1}^{I}\omega_{i,k,\ell}\exp(\theta_{i,k,\ell}){\phi}'_k(x_{i,k,\ell}){\phi}'_{\ell}(x_{i,k,\ell}),\]
where
\[ \theta_{i,k,\ell}= \sum_{j=1}^{s}y_j\frac{\sin(2\pi jx_{i,k,\ell})}{j^{2}}, \]
for the case $|k-\ell|\leq 1$. As stated in Section~3.2 of \cite{DKLS15}, computing $\theta_{i,k,\ell}$ for all $1\leq i\leq I$ and all $N$ Monte Carlo samples on $(y_1,\ldots,y_s)$ 
can be done in a fast way by using the TMC estimator, requiring $O(INM\log s)$ arithmetic operations. 
On the other hand, we need $O(INMs)$ arithmetic operations when using the standard Monte Carlo estimator.
This way, the log-normal case poses a further computational challenge compared to the uniform case, so that it would be interesting to see whether the TMC is still effective.
In this paper we apply the 3-point closed Newton-Cotes formula with nodes at $x_{k-1},x_k,x_{k+1}$ if $k=\ell$ and the 2-point closed one with nodes at $x_{(k+\ell-1)/2},x_{(k+\ell+1)/2}$ if $|k-\ell|=1$ .
Again we estimate the expectation of $u(1/2,\cdot)$.

Table~\ref{tbl:pde2} shows the results for various values of $N,M$ and $s$.
Similarly to the uniform case, we see that the TMC estimator is unbiased as the mean values agree well with the results for the stdMC estimator.
In this case, however, the variances for both of the estimators do not necessarily decay with the rate $O(1/N)$. This is possible because we increase $M$ and $s$ simultaneously with $N$, which may lead to an increment of the variance of $u_M(1/2, \cdot)$ in a non-asymptotic range of $N$.
As $N$ increases further, it is expected that the variance of $u_M(1/2, \cdot)$ stays almost the same and that the variances for both of the estimators tend to decay with the rate $O(1/N)$.
Moreover, it can be seen from the table that the magnitude of the variance for the TMC estimator is comparable to that of the stdMC estimator for many choices of $N, M, s$. 
As expected, the computational time for the stdMC estimator increases with $s$ significantly faster than the TMC estimator,
and it is clear that computation of the stiffness matrix takes most of the computational time, even for the log-normal case.
The relative efficiency of the TMC estimator over the stdMC estimator gets larger as $N$ (or, equivalently $s$) increases.

\begin{table*}
\caption{Estimating the expectation of $u(1/2,\cdot)$ with various values of $N, M$ and $s$ using the standard Monte Carlo method and the TMC method for the log-normal case.}\label{tbl:pde2}
\begin{center}
\begin{tabular}{ |c|c|c|c|c|c|c|c| } 
 \hline
  & \multicolumn{3}{|c|}{stdMC} & \multicolumn{3}{|c|}{TMC} & \\ \hline \hline
  & mean & variance & time &  mean & variance & time & efficiency \\ \hline 
 $N$ & \multicolumn{7}{|c|}{$N=M=s$} \\ \hline 
 64 & 0.018 & 8.34$\cdot 10^{-8}$ & 0.073 & 0.018 & 2.97$\cdot 10^{-7}$ & 0.079 & 0.260 \\ 
 128 & 0.018 & 1.17$\cdot 10^{-8}$ & 0.580 & 0.018 & 1.76$\cdot 10^{-7}$ & 0.171 & 0.227 \\
 256 & 0.018 & 1.55$\cdot 10^{-8}$ & 4.650 & 0.018 & 7.74$\cdot 10^{-8}$ & 0.549 & 1.694 \\
 512 & 0.018 & 2.11$\cdot 10^{-8}$ & 36.640 & 0.018 & 3.35$\cdot 10^{-8}$ & 2.260 & 10.230 \\
1024 & 0.018 & 2.65$\cdot 10^{-8}$ & 310.415 & 0.018 & 6.63$\cdot 10^{-8}$ & 10.385 & 11.956 \\
2048 & 0.018 & 1.23$\cdot 10^{-8}$ & 2391.495 & 0.018 & 3.06$\cdot 10^{-8}$ & 41.504 & 23.262 \\
4096 & 0.018 & 2.33$\cdot 10^{-8}$ & 20739.823 & 0.018 & 2.03$\cdot 10^{-8}$ & 182.282 & 131.012 \\ \hline
 $N$ & \multicolumn{7}{|c|}{$N=M^2=s$} \\ \hline 
256 & 0.017 & 8.08$\cdot 10^{-9}$ & 0.301 & 0.017 & 1.24$\cdot 10^{-8}$ & 0.032 & 6.069 \\
512 & 0.018 & 1.44$\cdot 10^{-8}$ & 1.680 & 0.018 & 9.75$\cdot 10^{-9}$ & 0.093 & 26.813 \\
1024 & 0.018 & 2.68$\cdot 10^{-8}$ & 9.762 & 0.018 & 1.24$\cdot 10^{-8}$ & 0.277 & 76.387 \\
2048 & 0.018 & 8.59$\cdot 10^{-9}$ & 56.480 & 0.018 & 5.37$\cdot 10^{-8}$ & 0.615 & 14.695 \\
4096 & 0.018 & 1.89$\cdot 10^{-8}$ & 317.294 & 0.018 & 2.98$\cdot 10^{-8}$ & 2.212 & 91.011 \\
8192 & 0.018 & 6.15$\cdot 10^{-8}$ & 1865.110 & 0.018 & 3.27$\cdot 10^{-8}$ & 7.094 & 494.909 \\
16384 & 0.018 & 1.91$\cdot 10^{-8}$ & 9964.289 & 0.018 & 2.35$\cdot 10^{-8}$ & 16.548 & 473.083 \\ \hline
 $N$ & \multicolumn{7}{|c|}{$N=2M=2s$} \\ \hline 
64 & 0.018 & 1.16$\cdot 10^{-7}$ & 0.021 & 0.018 & 1.83$\cdot 10^{-7}$ & 0.046 & 0.288 \\
128 & 0.018 & 9.58$\cdot 10^{-8}$ & 0.163 & 0.018 & 1.46$\cdot 10^{-7}$ & 0.126 & 0.844 \\
256 & 0.018 & 5.09$\cdot 10^{-8}$ & 1.279 & 0.018 & 7.30$\cdot 10^{-8}$ & 0.394 & 2.264 \\
512 & 0.018 & 2.40$\cdot 10^{-8}$ & 9.997 & 0.018 & 3.33$\cdot 10^{-8}$ & 1.089 & 6.616 \\
1024 & 0.018 & 9.53$\cdot 10^{-8}$ & 82.822 & 0.018 & 6.00$\cdot 10^{-8}$ & 4.914 & 26.758 \\
2048 & 0.018 & 1.18$\cdot 10^{-8}$ & 644.228 & 0.018 & 2.97$\cdot 10^{-8}$ & 19.265 & 13.243 \\
4096 & 0.018 & 1.45$\cdot 10^{-8}$ & 5127.910 & 0.018 & 1.72$\cdot 10^{-8}$ & 73.837 & 58.506 \\ \hline
 $N$ & \multicolumn{7}{|c|}{$N=2M^2=2s$} \\ \hline 
512 & 0.017 & 1.81$\cdot 10^{-8}$ & 0.597 & 0.017 & 5.08$\cdot 10^{-9}$ & 0.060 & 35.776 \\
1024 & 0.018 & 5.07$\cdot 10^{-9}$ & 3.370 & 0.018 & 6.70$\cdot 10^{-9}$ & 0.151 & 16.834 \\
2048 & 0.018 & 9.32$\cdot 10^{-9}$ & 22.059 & 0.018 & 7.58$\cdot 10^{-9}$ & 0.393 & 68.996 \\
4096 & 0.018 & 8.18$\cdot 10^{-9}$ & 123.556 & 0.018 & 5.70$\cdot 10^{-9}$ & 1.187 & 149.507 \\
8192 & 0.018 & 4.32$\cdot 10^{-9}$ & 646.786 & 0.018 & 4.16$\cdot 10^{-9}$ & 4.738 & 141.992 \\
16384 & 0.018 & 2.47$\cdot 10^{-9}$ & 3546.950 & 0.018 & 5.11$\cdot 10^{-9}$ & 11.673 & 146.925 \\
32768 & 0.018 & 1.81$\cdot 10^{-9}$ & 20018.280 & 0.018 & 3.97$\cdot 10^{-9}$ & 33.323 & 274.032 \\ \hline
\end{tabular}
\end{center}
\end{table*}

\subsection{2D differential equation with random coefficients}
Following Section~4 of \cite{DKLS16}, our last example considers the following two-dimensional ODE with the uniform random coefficients:
\begin{align*}
& -\nabla \cdot \left( a(\bsx,\bsy)\nabla u(\bsx,\bsy)\right) = g(\bsx)\equiv 100 x_1 \\
& \qquad \qquad \qquad \text{for $\bsx\in [0,1]^2$ and $\bsy\in \left[ -\frac{1}{2}, \frac{1}{2}\right]^{\NN}$} \\
& u(\bsx,\bsy) = 0\quad \text{for $\bsx\in \partial\left( [0,1]^2\right)$} \\
& a(\bsx,\bsy) = 1+\sum_{j=1}^{\infty}y_j\frac{\sin(\pi k_{j,1}x_1)\sin(\pi k_{j,2}x_2)}{(k_{j,1}^2+k_{j,2}^2)^2}.
\end{align*}
Here the elements $(k_{j,1},k_{j,2})_j$ are ordered in such a way that $\{(k_{j,1},k_{j,2})\mid j\in \NN \} = \NN\times \NN$ and
\[ \frac{1}{(k_{j,1}^2+k_{j,2}^2)^2}\geq \frac{1}{(k_{j+1,1}^2+k_{j+1,2}^2)^2}\quad \text{for all $j\in \NN$.}\]
In cases where equality holds, the ordering is arbitrary. 
We solve this ODE by a finite element discretization.
Given the boundary condition on $u$, we exclude the basis functions along the boundary
and use a set of local piecewise linear hat functions $\{\phi_{p,q}: 0< p,q< M\}$ as the system of basis functions.
The basis function $\phi_{p,q}$ has center at $(p/M,q/M)$ and the support is given by the following hexagon:
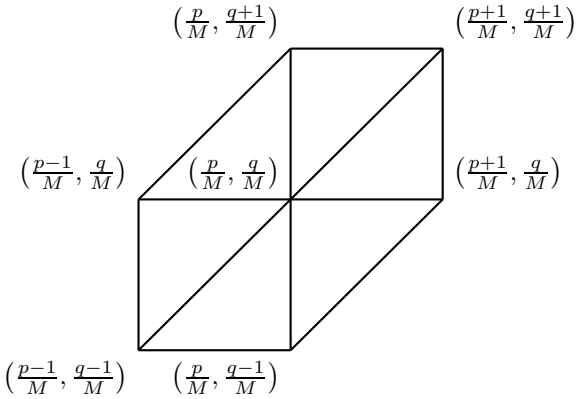
\begin{center}
\begin{tikzpicture}[domain=-4:4, very thick]
\draw[thick] (0,0)--(2,0);
\path (0,0) node[above left] {$\left(\frac{p}{M},\frac{q}{M}\right)$};
\path (2,0) node[above right] {$\left(\frac{p+1}{M},\frac{q}{M}\right)$};
\draw[thick] (0,0)--(-2,0);
\path (-2,0) node[above left] {$\left(\frac{p-1}{M},\frac{q}{M}\right)$};
\draw[thick] (0,0)--(0,2);
\path (0,2) node[above left] {$\left(\frac{p}{M},\frac{q+1}{M}\right)$};
\draw[thick] (0,0)--(0,-2);
\path (0,-2) node[below left] {$\left(\frac{p}{M},\frac{q-1}{M}\right)$};
\draw[thick] (0,0)--(2,2);
\path (2,2) node[above right] {$\left(\frac{p+1}{M},\frac{q+1}{M}\right)$};
\draw[thick] (0,0)--(-2,-2);
\path (-2,-2) node[below left] {$\left(\frac{p-1}{M},\frac{q-1}{M}\right)$};
\draw[thick] (2,0)--(2,2);
\draw[thick] (0,2)--(2,2);
\draw[thick] (-2,-2)--(-2,0);
\draw[thick] (-2,-2)--(0,-2);
\draw[thick] (-2,0)--(0,2);
\draw[thick] (0,-2)--(2,0);
\end{tikzpicture}
\end{center}

The approximate solution $u_M$ of the ODE in this setting is given by
\[ u_M = \sum_{p,q=1}^{M-1}\hat{u}_{p,q} \phi_{p,q}(x_1,x_2), \]
with
\begin{align}
& \left( \hat{u}_{1,1},\ldots,\hat{u}_{1,M-1},\ldots,\hat{u}_{M-1,1},\ldots,\hat{u}_{M-1,M-1}\right)\notag \\
& \quad \cdot B(y_1,\ldots,y_s) \notag \\
& = (\hat{g}_{1,1},\ldots,\hat{g}_{1,M-1},\ldots,\hat{g}_{M-1,1},\ldots,\hat{g}_{M-1,M-1}), \label{eq:linear_system2} \end{align}
for the stiffness matrix $B$ depending on $y_1,\ldots,y_s$ and the forcing vector with entries $\hat{g}_{p,q}=\int_{[0,1]^2} g(\bsx)\phi_{p,q}(\bsx)\rd \bsx$,
which is explicitly computable.
By truncating the infinite sum appearing in the random field $a$ to the first $s$ terms, 
the entries of the matrix $B=(b_{(p,q),(p',q')})_{p,q,p',q'}$ are given by
\begin{align*}
 & b_{(p,q),(p',q')} \\
 & = \int_{[0,1]^2} a(\bsx,(y_1,\ldots,y_s,0,\ldots)) \nabla \phi_{p,q}(\bsx)\nabla \phi_{p',q'}(\bsx)\rd \bsx \\
 & = \int_{[0,1]^2}\nabla \phi_{p,q}(\bsx)\nabla \phi_{p',q'}(\bsx)\rd \bsx \\
 & \quad +\sum_{j=1}^{s}y_j\int_{[0,1]^2}\frac{\sin(\pi k_{j,1}x_1)\sin(\pi k_{j,2}x_2)}{(k_{j,1}^2+k_{j,2}^2)^2} \\
 & \qquad \qquad \qquad \times \nabla \phi_{p,q}(\bsx)\nabla \phi_{p',q'}(\bsx)\rd \bsx\\
 & =: a_{(p,q),(p',q')}^{(0)}+\sum_{j=1}^{s}y_j a_{(p,q),(p',q')}^{(j)}.
\end{align*}
Similarly to the one-dimensional case, every term $a_{(p,q),(p',q')}^{(j)}$ is explicitly calculable.
Hence, by letting $A^{(j)} = (a_{(p,q),(p',q')}^{(j)})_{p,q,p',q'}$ for $j=0,1,\ldots,s$, we have
\[ B=A^{(0)}+\sum_{j=1}^{s}y_jA^{(j)}. \]

This time, unlike the one-dimensional uniform case, each matrix $A^{(j)}$ is no longer tridiagonal but has a bandwidth of $O(M)$.
Instead of the LU decomposition, we use the BiCGSTAB method without preconditioner to solve the system of linear equations \eqref{eq:linear_system2}. 
We refer to \cite{Sbook} for detailed information on the BiCGSTAB method.
Since this is an iterative method, the resulting solution $u_M$ is not precise and the required computational time depends on the stopping criterion we use.
Moreover, such an additional computational burden makes it a priori unclear whether TMC can significantly reduce the computational time as a whole.
However, in all our tests, computing the matrix $B$ for $N$ Monte Carlo samples on $(y_1,\ldots,y_s)$ 
remained the computationally dominant part, and the relative efficiency of the TMC estimator did not strongly depend on the stopping criterion.
For our experiments, we estimate the expectation of $u((1/2,1/2),\cdot)$.

Table~\ref{tbl:pde3} shows the results for various values of $N,M$ and $s$.
The stopping criterion of the BiCGSTAB method is set such that the relative 2-norm of the residual is less than $10^{-5}$.
We see that the TMC estimator is unbiased, and that the variances for both of the estimators approximately decay with the rate $O(1/N)$.
Similarly to the one-dimensional log-normal case, the magnitude for the TMC estimator is comparable to that of the stdMC estimator.
The computational time for the stdMC estimator increases with $s$ much faster than the TMC estimator, 
resulting in a substantial relative efficiency of the TMC estimator over the stdMC estimator for larger $s$.

\begin{table*}
\caption{Estimating the expectation of $u((1/2,1/2),\cdot)$ with various values of $N, M$ and $s$ using the standard Monte Carlo method and the TMC method for the two-dimensional uniform case.}\label{tbl:pde3}
\begin{center}
\begin{tabular}{ |c|c|c|c|c|c|c|c| } 
 \hline
  & \multicolumn{3}{|c|}{stdMC} & \multicolumn{3}{|c|}{TMC} & \\ \hline \hline
  & mean & variance & time &  mean & variance & time & efficiency \\ \hline 
 $N$ & \multicolumn{7}{|c|}{$N=M^2=s$} \\ \hline 
 16 & 3.516 & 1.09$\cdot 10^{-3}$ & 0.0002 & 3.514 & 7.17$\cdot 10^{-4}$ & 0.0020 & 0.119 \\ 
 64 & 3.647 & 2.16$\cdot 10^{-4}$ & 0.009 & 3.643 & 2.43$\cdot 10^{-4}$ & 0.027 & 0.302 \\ 
 256 & 3.676 & 8.03$\cdot 10^{-5}$ & 1.248 & 3.678 & 8.84$\cdot 10^{-5}$ & 0.406 & 2.792 \\
1024 & 3.685 & 1.89$\cdot 10^{-5}$ & 127.366 & 3.686 & 1.25$\cdot 10^{-5}$ & 8.369 & 22.925 \\
4096 & 3.688 & 3.62$\cdot 10^{-6}$ & 14773.594 & 3.688 & 3.74$\cdot 10^{-6}$ & 228.490 & 62.562 \\ \hline
 $N$ & \multicolumn{7}{|c|}{$N=2M^2=2s$} \\ \hline 
 32 & 3.517 & 3.97$\cdot 10^{-4}$ & 0.0004 & 3.515 & 3.20$\cdot 10^{-4}$ & 0.0035 & 0.132 \\ 
 128 & 3.646 & 6.81$\cdot 10^{-5}$ & 0.019 & 3.646 & 1.26$\cdot 10^{-4}$ & 0.057 & 0.174 \\ 
 512 & 3.676 & 3.41$\cdot 10^{-5}$ & 2.773 & 3.677 & 3.35$\cdot 10^{-5}$ & 0.769 & 3.670 \\
2048 & 3.686 & 5.67$\cdot 10^{-6}$ & 260.868 & 3.686 & 7.59$\cdot 10^{-6}$ & 16.375 & 11.893 \\
8192 & 3.688 & 1.55$\cdot 10^{-6}$ & 47914.800 & 3.688 & 1.66$\cdot 10^{-6}$ & 494.336 & 90.562 \\ \hline
 $N$ & \multicolumn{7}{|c|}{$2N=M^2=2s$} \\ \hline 
 8 & 3.517 & 1.83$\cdot 10^{-3}$ & 0.0001 & 3.513 & 1.89$\cdot 10^{-3}$ & 0.0021 & 0.046 \\ 
 32 & 3.646 & 4.06$\cdot 10^{-4}$ & 0.003 & 3.639 & 3.19$\cdot 10^{-4}$ & 0.020 & 0.194 \\ 
 128 & 3.679 & 7.52$\cdot 10^{-5}$ & 0.291 & 3.679 & 1.26$\cdot 10^{-4}$ & 0.771 & 0.771 \\
 512 & 3.685 & 2.61$\cdot 10^{-5}$ & 21.702 & 3.686 & 3.35$\cdot 10^{-5}$ & 4.135 & 4.087 \\
2048 & 3.687 & 3.63$\cdot 10^{-6}$ & 2385.625 & 3.688 & 7.59$\cdot 10^{-6}$ & 105.021 & 10.862 \\ \hline
 $N$ & \multicolumn{7}{|c|}{$2N=M^2=s$} \\ \hline 
 32 & 3.522 & 4.11$\cdot 10^{-4}$ & 0.0005 & 3.515 & 3.21$\cdot 10^{-4}$ & 0.0035 & 0.184 \\ 
 128 & 3.646 & 7.54$\cdot 10^{-5}$ & 0.034 & 3.646 & 1.27$\cdot 10^{-4}$ & 0.041 & 0.490 \\ 
 512 & 3.676 & 2.61$\cdot 10^{-5}$ & 5.259 & 3.677 & 3.35$\cdot 10^{-5}$ & 0.749 & 5.468 \\
2048 & 3.685 & 3.63$\cdot 10^{-6}$ & 742.863 & 3.686 & 7.59$\cdot 10^{-6}$ & 16.304 & 21.790 \\
8192 & 3.688 & 1.96$\cdot 10^{-6}$ & 115052.392 & 3.688 & 1.66$\cdot 10^{-6}$ & 510.611 & 266.587 \\ \hline
\end{tabular}
\end{center}
\end{table*}

\section{Conclusion}
Motivated by applications to partial differential equations with random coefficients, we introduced the Toeplitz Monte Carlo estimator in this paper.
The theoretical analysis of the TMC estimator shows that it is unbiased and the variance converges with the canonical $1/N$ rate.
From the viewpoint of tractability in the weighted $L_2$ space, the TMC estimator requires a stronger condition on the weights than the standard Monte Carlo estimator to achieve strong polynomial tractability.
Through a series of numerical experiments for PDEs with random coefficients, we observed that the TMC estimator is quite effective in reducing
necessary computational times and the relative efficiency over the standard Monte Carlo estimator is substantial, particularly for high-dimensional settings.

We leave the following topics open for future research.
\begin{itemize}
\item Combination with variance reduction techniques: 
In our numerical experiments for the one-dimensional uniform case, the variance of the TMC estimator tends to be much larger than the standard Monte Carlo estimator.
To address this issue, it would be reasonable to consider applying some variance reduction techniques to the TMC estimator 
such that the resulting algorithm still allows for a fast matrix-vector multiplication. In particular, it would be interesting to design a variance reduction technique which reduces the term
\[ \sum_{\ell=1}^{\min(s,N)-1}(N-\ell)\sum_{\emptyset \neq u\subseteq [1:s-\ell]}I_{\rho_{|u|}}(f_uf_{u+\ell}).\]
\item Multilevel Toeplitz Monte Carlo (MLTMC): 
Recently, multilevel Monte Carlo (MLMC) methods from \cite{G08} have been studied intensively in the context of PDEs with random coefficients, see for instance \cite{CGST11} and \cite{TRGU13}.
By combining the TMC estimator with MLMC, the dependence of the total computational complexity not only on the truncation dimension $s$ 
but also on the discretization parameter $M$ can be possibly weakened.
\item Applications to different areas:
Although this work has been originally motivated by PDEs with random coefficients, the TMC estimator itself is more general and can be 
applied in different contexts as well. Since generating points from multivariate normal distribution is quite common, 
for instance, in financial engineering, operations research and machine learning, one may apply the TMC estimator also to those areas.
\end{itemize}

\begin{acknowledgements}
Josef Dick is partly supported by the Australian Research Council Discovery Project DP190101197. Takashi Goda is partly supported by JSPS KAKENHI Grant Number 20K03744. J.D. would like to thank T.G. for his hospitality during his visit at U. Tokyo.
\end{acknowledgements}



\end{document}